\documentclass[11pt]{amsart}
\usepackage{amsmath, graphicx, amsfonts, amssymb}
\usepackage{amsfonts, mathrsfs, color, amsthm, amsbsy}
\usepackage{mathtools}
\usepackage{dsfont} 
\usepackage{adjustbox}
\usepackage{enumerate}
\usepackage[normalem]{ulem}
\usepackage{amsrefs}
\usepackage{comment}
\usepackage{soul}
\usepackage{comment}
\usepackage[utf8]{inputenc}

\usepackage{geometry} 
\usepackage{parskip}
\newtheorem{thmx}{Theorem}
\usepackage{hyperref}

\usepackage[symbol]{footmisc}

\usepackage{tikz}
\usetikzlibrary{positioning,patterns,calc}

\newcommand{\real}{\mathbb{R}} 
\newcommand{\R}{\mathbb{R}} 
\newcommand{\integer}{\mathbb{Z}} 

\newcommand{\nball}{B^{N}} 
\newcommand{\eball}{B} 
\renewcommand{\O}{\mathcal{O}} 
\renewcommand{\S}{\mathbb{S}} 
\newcommand{\F}{\mathcal{F}} 
\newcommand{\jfg}{J_g^\F} 

\newcommand{\haus}{\mathcal{H}} 
\renewcommand{\H}{\mathcal{H}} 

\newcommand{\csvf}{C^1_c} 

\newcommand{\graph}{\textup{graph}} 
\newcommand{\lploc}[1]{\mathbf{L}_{\textup{loc}}^{#1}} 

\newcommand{\ivarifolds}{\mathcal{IV}} 
\newcommand{\setv}{\mathbf{v}} 
\newcommand{\supp}{\textnormal{spt}} 
\newcommand{\spt}{\textnormal{spt}} 
\newcommand{\restrictv}{\mathbin{\hspace{0.1em}\vrule height 1.3ex depth 0pt width 0.13ex\vrule height 0.13ex depth 0pt width 1.0ex}} 
\newcommand{\mass}[1]{\|#1\|} 
\newcommand{\sclass}{\mathfrak{s}} 

\newcommand{\reg}{\textup{reg}} 
\newcommand{\sing}{\textup{sing}} 
\newcommand{\regscale}{\pmb{r}} 
\newcommand{\badreg}{\mathcal{B}} 
\newcommand{\genreg}{\textnormal{gen-reg}} 

\newcommand{\strata}{\mathcal{S}} 
\newcommand{\kcone}{\mathscr{C}} 

\newcommand{\BC}{\mathbf{C}}
\newcommand{\del}{\partial} 

\newcommand{\indexa}{\textnormal{index}_a}
\newcommand{\windexa}{\omega\textnormal{-index}_a} 
\newcommand{\inj}{\textnormal{inj}} 
\renewcommand{\exp}{\textnormal{exp}} 
\newcommand{\vol}{\textup{vol}} 

\newtheorem{theorem}{Theorem}[section]
\newtheorem{lemma}[theorem]{Lemma}

\newtheorem{corollary}[theorem]{Corollary}

\newtheorem{definition}[theorem]{Definition}
\newtheorem*{theorem*}{Theorem}

\makeatletter
\@addtoreset{claim}{theorem}
\makeatother

\newtheoremstyle{newtheoremstyledefn}
{3pt}
{3pt}
{}
{\parindent}
{\bfseries}
{.}
{0.5em}
{} 

\theoremstyle{newtheoremstyledefn}
\newtheorem{remark}[theorem]{Remark}

\usepackage{etoolbox}
\makeatletter
\patchcmd{\@maketitle}
  {\ifx\@empty\@dedicatory}
  {\ifx\@empty\@date \else {\vskip3ex \centering\footnotesize\@date\par\vskip1ex}\fi
   \ifx\@empty\@dedicatory}
  {}{}
\patchcmd{\@adminfootnotes}
  {\ifx\@empty\@date\else \@footnotetext{\@setdate}\fi}
  {}{}{}
\makeatother

\title[Quantitative Estimates for PMC Hypersurfaces]{Quantitative Estimates on the Topology and Singular Set of Prescribed Mean Curvature Hypersurfaces}
\author{Nicolau S. Aiex, Sean McCurdy, and Paul Minter}
\address{Department of Mathematics, National Taiwan Normal University, Taipei, Taiwan}
\email{nsarquis@math.ntnu.edu.tw}
\address{Instituto de Matemáticas, Universidad Nacional Autónoma de México, Ciudad Universitaria, 04510,
México, CDMX}
\email{sean.mccurdy@im.unam.mx}
\address{Department of Pure Mathematics and Mathematical Statistics, University of Cambridge}
\email{pdtwm2@cam.ac.uk}

\begin{document}

\begin{abstract}
We establish quantitative topological and singularity properties for (certain) prescribed mean curvature (PMC) hypersurfaces $V^n$ in Riemannian manifolds $(N^{n+1},h)$. Indeed, if $
V$ has area at most $A>0$ with PMC given by a $C^{1,\alpha}$ function $g:N\to \R$ with the bound $|g|_{C^{1,\alpha}}\leq \Gamma$, we show that there exists a constant $C$ depending only on $n,h,A,\Gamma$ and geometric quantities such that:
\begin{align*}
    \sum^n_{i=0}b^i(V) & \leq C(1+\text{index}(V)) \quad \text{if }3\leq n+1\leq 7;\\
    \mathcal{M}^{*n-7}(\sing(V)) & \leq C(1+\text{index}(V)) \quad \text{if }n+1\geq 8.
\end{align*}
Here, $b^i$ denote the Betti numbers over any field, $\mathcal{M}^{*n-7}$ denotes the upper $(n-7)$-dimensional Minkowski content, and $\sing(V)$ is the singular set of $V$. The first inequality extends the work of Song from the minimal hypersurface setting to the PMC hypersurface setting, whilst the second extends work of the authors. Our results apply to the PMC hypersurfaces constructed recently through min-max techniques by Bellettini--Wickramasekera.
\end{abstract}

\maketitle

\section{Introduction}

In \cite{song.a2023}, A.~Song established quantitative estimates on both the topological data (in ambient dimensions $3\leq n+1\leq 7$) and the $(n-7)$-dimensional Hausdorff measure of the singular set (in ambient dimensions $n+1\geq 8$) of minimal hypersurfaces $V^{n}$ in a Riemannian manifold $(N^{n+1},h)$ in terms of the area and Morse index of the minimal hypersurface. In \cite{aiex-mccurdy-minter2024}, the authors improved on the result concerning the singular set when $n+1\geq 8$ by establishing bounds on its upper $(n-7)$-dimensional Minkowski content. These results on the singular set in turn extend work of Naber--Valtorta \cite{naber-valtorta2020} from area minimizing hypersurfaces to minimal hypersurfaces with finite index.

Once one has certain a priori estimates, the arguments in \cite{song.a2023} are of a combinatorial nature. Song suggests that, since many arguments in \cite{song.a2023} are not specific to minimal hypersurfaces (and similarly for \cite{aiex-mccurdy-minter2024}), they may apply to other variational objects. The purpose of this short note is to illustrate that this is possible for (certain) hypersurfaces with prescribed mean curvature (PMC) and finite index.

We characterize PMC hypersurfaces in a Riemannian manifold $(N^{n+1},h)$ as solutions to a variational problem arising from a parametric integrand depending on a given $C^{1,\alpha}$ function $g:N\to \real$, as in \cite{schoen-simon1981}*{$(1.2)$--$(1.4)$} and \cite{bellettini-wickramasekera2019:arxiv}*{Definition 6.1}. Unlike minimal hypersurfaces, PMC hypersurfaces may naturally have tangential self-intersections depending on whether $g$ has a sign or not (for a given choice of unit normal). As such, the `genuine' singular set may not directly correspond to the singular set of the varifold representing the PMC hypersurface. In fact, we may decompose the support of such a varifold $V$ as
$$\spt\|V\| = \genreg\,V\cup\Sigma,$$
where $\genreg\, V$ is an immersed hypersurface of class $C^2$ with only tangential self-intersections and $\Sigma$ is the remaining `genuine' singular set.

The proofs in \cite{song.a2023,aiex-mccurdy-minter2024} utilise that minimal hypersurfaces with finite index are locally stable. We will utilise an analogous phenomenon here, although one must be somewhat careful with which notion of index is used in the PMC setting when the prescribing function $g$ does not have a sign. Suitable covering arguments based on the stability radius are then used to prove the results, with a key point being to control the number of balls used in the cover appropriately by the index in order to deduce the claimed bounds. In the minimal hypersurface setting, the regularity theory of stable minimal hypersurfaces \cite{schoen-simon1981, wickramasekera2014} is of crucial importance. Recent work by Bellettini--Wickramasekera \cite{bellettini-wickramasekera2019:arxiv} establishes the corresponding regularity theory for PMC hypersurfaces. This will allow us to use analogous arguments to those in the minimal hypersurface case.

Our results are divided into the low-dimensional case $3\leq n+1\leq 7$ (where the set $\Sigma$ above is necessarily empty) and the high-dimensional case $n+1\geq 8$. In the low-dimensional case where $V = \genreg\,V$, we fix a field $\mathbb{F}$ and write $b^i(\cdot) = b^i(\cdot,\mathbb{F})$ for the dimension over $\mathbb{F}$ of the cohomology groups $H^i(\cdot,\mathbb{F})$. In the high-dimensional case, we recall from \cite{bellettini-wickramasekera2019:arxiv} that necessarily $\dim_\H(\Sigma)\leq n-7$.

\begin{thmx}\label{thm:A-0}
    Let $3\leq n+1\leq 7$, $\Gamma,\Lambda,\mu,\mu_1>0$, and $I\in \{0,1,2,\dotsc\}$. Let $(N^{n+1},h)$ be a closed Riemannian manifold and fix a $C^{1,\alpha}$ function $g:N\to \R$ with $|g|_{1,\alpha}\leq\Gamma$. Then if $V\in\sclass_{g,\Lambda,I,\mu,\mu_1}(N)$, one has
    $$\sum^n_{i=0}b^i(V) \leq C(1+I).$$
    Here, $C = C(N,h,\Gamma,\Lambda,\mu,\mu_1)\in (0,\infty)$.
\end{thmx}

\begin{thmx}[cf.~Theorem \ref{main theorem n>7} and Theorem \ref{main theorem n=7}]\label{thm:A}
Let $n+1\geq 8$, $I\in\{0,1,2,\dotsc\}$, $\Lambda,\Gamma,\mu,\mu_1\in(0,\infty)$, $K\in(0,1/2)$, and $\alpha\in(0,1]$.
Suppose $(N^{n+1},h)$ is a Riemannian manifold with $0\in N$, $\left|\left.\textnormal{sec}\right|_{B^N_2}\right| \leq K$, $\left.\inj\right|_{B^N_2}\geq K^{-1}$, and $g:\nball_2\rightarrow\real$ is a function of class $C^{1,\alpha}$.
If $|g|_{1,\alpha}\leq \Gamma$ and $V\in\sclass_{g,\Lambda,I,\mu,\mu_1}(\nball_2)$ then $\Sigma=\spt\mass{V}\setminus\genreg\,V$ is countably $(n-7)$-rectifiable and we have, for any $0<r\leq 1/2$:
\begin{equation*}
\begin{aligned}
\haus^{n+1}\left(\nball_{r/8}\left(\Sigma\right)\cap\nball_{1/2}\right)\leq C_0(1+I)r^8;\\
\mass{V}\left(\nball_{r/8}\left(\Sigma\right)\cap\nball_{1/2}\right)\leq C_0(1+I)r^7.
\end{aligned}
\end{equation*}
In particular, $\mathcal{M}^{*n-7}(\Sigma\cap B^N_{1/2})\leq C_0(1+I)$, where $C_0=C_0(n,\mu,\mu_1,\Lambda,K,\Gamma,\alpha)\in(0,\infty)$.
\end{thmx}

\textbf{Note:} Theorem \ref{thm:A} is a local result. If $N$ is additionally closed, it can be improved to a global result as in \cite[Theorem C]{aiex-mccurdy-minter2024}.

The class $\sclass_{g,\Lambda,I,\mu,\mu_1}(\nball_2)$ roughly consists of varifolds in $\nball_2$ with prescribed mean curvature $g$, index bounded by $I$, and mass upper bound $\Lambda$ (see Definition \ref{admissible varifolds} for the precise definition); $\mu,\mu_1$ are parameters controlling the PMC functional. These classes include, for instance the PMC hypersurfaces recently constructed by Bellettini--Wickramasekera \cite{BW20}. In comparison to the minimal case, note that the constant $C_0$ now further depends on (a bound for) the prescribing function $g$. Since the constant $C_0$ includes the Naber--Valtorta constant from \cite{naber-valtorta2020}, it remains unclear what its explicit dependency on both the mass of the varifold and the bounds on the mean curvature are. 

When $n+1=3$, it may be possible to show that the dependence on $\Lambda$ can be taken to be linear, similarly to the minimal case shown in \cite{song.a2023}. At least in the CMC case, one would need to prove a quantitative version of Meeks--Perez--Ros removable singularity theorem for CMC surfaces \cite{meeks-perez-ros2016}*{Theorem 1.2} as was done in the minimal case in \cite{song.a2023}*{Theorem 14}, which in turn depends on the corresponding area-independent curvature estimates in \cite{meeks-perez-ros2016}*{Theorem 1.3}.
The case when the mean curvature is allowed to change sign is less clear, as we are not aware of any area-independent curvature estimates in the literature yet.

The proofs of Theorem \ref{thm:A-0} and Theorem \ref{thm:A} follow closely those used in the corresponding arguments in \cite{song.a2023,aiex-mccurdy-minter2024} once one has the correct set-up and regularity results in the PMC setting. Our focus will therefore be on these differences. Since the core of the argument in \cite{aiex-mccurdy-minter2024} is presented when the ambient space is Euclidean, here we present the full proof directly in the Riemannian setting. For Theorem \ref{thm:A-0} and the $n+1=8$ case of Theorem \ref{thm:A} (where one cannot argue as in \cite{aiex-mccurdy-minter2024} due to a lack of scaling of the measure), we will briefly detail the necessary (minor) adaptations of Song's argument to the PMC setting.

\textbf{Organization:} In Section \ref{sec:prelim} we give the precise description of the class of admissible varifolds with prescribed mean curvature which our theorem applies to, including details regarding the decomposition of the singular set and facts about such varifolds with bounded index and the stability radius. In Section \ref{sec:quantitative} we then briefly review the quantitative stratification of the singular set of such a varifold as well as Naber--Valtorta's main theorem. In Section \ref{sec:main theorem} we prove Theorem \ref{thm:A} when $n+1\geq 9$. In Section \ref{case n+1=8} we include the modifications to Song's work \cite{song.a2023} needed to prove Theorem \ref{thm:A} when $n+1=8$ and Theorem \ref{thm:A-0}.

\textbf{Acknowledgments:} NSA was supported through the grant with No.~113-2115-M-003-001-MY2 by the National Science and Technology Council. This research was conducted during the period that PM was a Clay Research Fellow.

\section{Preliminaries}\label{sec:prelim}

We fix throughout $\alpha\in(0,1]$, $n+1\geq 3$ a positive integer, and $(N,h)$ a smooth Riemannian manifold of dimension $n+1$ with Riemannian metric $h$. We write $\eball_r(p)\subset\real^{n+1}$ for the Euclidean ball of radius $r>0$ centered at $p\in\real^{n+1}$ and $\nball_r(x)\subset N$ for the intrinsic ball of radius $r$ centered at $x\in N$ defined by the distance induced by $h$.

Given an open subset $U\subset N$ we denote by $\ivarifolds_n(U)$ the space of integral $n$-varifolds in $U$. If $R\subset U$ is a $\haus^n$-rectifiable set and $\theta:U\rightarrow\real$ is a non-negative $\haus^n$-measurable function which is locally integrable on $U$ with respect to $\haus^n$ we denote the induced $n$-varifold of density $\theta$ by $\setv(R,\theta)$. For $p\in\real^{n+1}$ and $\lambda>0$ we also define $\eta_{p,\lambda}:\real^{n+1}\rightarrow\real^{n+1}$ to be the map $\eta_{p,\lambda}(x):=\lambda^{-1}(x-p)$.

\subsection{Prescribed Mean Curvature Varifolds}

Following \cite{bellettini-wickramasekera2019:arxiv}, we first give a variational description of immersed hypersurfaces with mean curvature prescribed by a given function.

Let $U\subset N$ be an open set. If $S$ is a $n$-dimensional manifold and $\iota:S\rightarrow U$ is a $C^1$ immersion then we identify sections of the vector bundle $\iota^*TN\rightarrow S$ with maps $S\rightarrow TN$ without explicitly writing the pushforward by the corresponding immersion, and similarly for vector bundles derived from $\iota^*TN$ (such as the normal bundle of $S$).

\begin{definition}\label{admissible sets and variations}
    Let $\iota:S\to U$ be a $C^1$ immersion as above. We say that $\O\subset U$ is an admissible open set if:
    \begin{enumerate}
        \item [\textnormal{(o1)}] $\O$ is an oriented open set with compact closure contained in $U$;
        \item [\textnormal{(o2)}] $S_\O := \iota^{-1}(\O)$ is orientable with compact closure in $S$.
    \end{enumerate}
    We then say that a $C^1$ map $\psi:(-\varepsilon,\varepsilon)\times S\to U$ is an admissible variation (here, $\varepsilon>0$), and that $(\O,\psi)$ is a variation pair, if in addition we have:
    \begin{enumerate}
        \item [\textnormal{(v1)}] $\psi(0,y) = \iota(y)$ for all $y\in S$;
        \item [\textnormal{(v2)}] $\psi_t(\cdot) := \psi(t,\cdot)$ is an immersion for all $t\in (-\varepsilon,\varepsilon)$;
        \item [\textnormal{(v3)}] $\psi_t(S_\O)\subset \O$ for all $t\in (-\varepsilon,\varepsilon)$;
        \item [\textnormal{(v4)}] $\spt\,\psi := \overline{S\setminus \{y\in S:\psi_t(y) = \iota(y) \text{ for all }t\in (-\varepsilon,\varepsilon)\}}$ is compact and contained in $S_\O$.
    \end{enumerate}
\end{definition}
Given a variation pair $(\O,\psi)$, we will always assume that we choose the orientations of $S_\O$ and $\O$ as well as sections $\nu_\O^t$ of the normal bundles of $\psi_t:S_\O\to U$ such that $d\psi_t(\vec{S}_\O)\wedge \nu_\O^t$ is positive with respect to the orientation of $\O$.
\begin{definition}
    Given a variation pair $(\O,\psi)$, we say that $\psi$ is an ambient deformation if
    $$h\left(\partial_t\psi_t\big|_{t=0},  \nu_\O^0(y)\right) = \phi(\iota(y))$$
    for some $\phi\in C^1_c(\O)$ and all $y\in S_\O$. In this case, we say that $\psi$ is induced by $\phi$.
\end{definition}
\begin{remark}
    When $S_\O$ is embedded, ambient deformations and normal deformations are in one-to-one correspondence. However, when $S_\O$ is not embedded then ambient deformations comprise a larger set of variations. Furthermore, in this situation deformations induced by the flow of an ambient vector field (i.e.~a vector field on $N$) are not considered ambient deformations as described above. Indeed, the ambient deformations above allow tangential self-intersections to be deformed onto local embeddings if the normal vector is opposite on a pair of components.
\end{remark}

We can now define the relative enclosed $g$-volume with respect to a variation pair.

\begin{definition}[\cite{bellettini-wickramasekera2019:arxiv}*{Definition 1.1}]
Suppose $g\in C^{1,\alpha}(N)$, $S$ is an immersed hypersurface in $N$, and $(\O,\psi)$ is a variation pair. We define the enclosed $g$-volume as
\begin{equation*}
\begin{aligned}
\vol_g(t) & = \int_{I_t\times S_\O}\psi^*(g\, d\vol_\O),\\
\end{aligned}
\end{equation*}
where $I_t=[0,t]$ if $t>0$ and $[t,0]$ if $t<0$. We say $\psi$ is $g$-volume preserving if $\vol_g(t)$ is constant.
\end{definition}

We include the definition of parametric elliptic functionals as in \cites{schoen-simon1981,bellettini-wickramasekera2019:arxiv}.

\begin{definition}[\cite{bellettini-wickramasekera2019:arxiv}*{Definition 6.1}]\label{elliptic functional}
Let $\mu,\mu_1,\rho>0$. We define the class $\tilde{\mathcal{I}}(\mu,\mu_1,{\rho})$ of parametric elliptic integrands as the functions $\tilde{F}:\eball_{\rho}(0)\times\real^{n+1}\setminus\{0\}\rightarrow\real$ satisfying:
\begin{enumerate}[(1)]
\item [\textnormal{(1)}] $z\mapsto\tilde{F}(z,p)$ is of class $C^{2,\alpha}$ for some $\alpha\in(0,1)$ for all $p\in\real^{n+1}\setminus\{0\}$ and $p\mapsto\tilde{F}(z,p)$ is of class $C^3$ for all $z\in\eball_\rho(0)$; and

\item [\textnormal{(2)}] For each $y\in\eball_{\rho/2}(0)$ there exists $v(y)\in\real^{n+1}$ and a $C^2$ diffeomorphism $\eta_y:\eball_\rho(0)\to \eball_\rho(0)$ such that $\eta_0=\textup{id}$, $\eta_y(0)=y$, and if $\hat{F}_{v(y)}(z,p)=\tilde{F}(z,p)+v(y)\cdot p$ then $\eta_y^\#\hat{F}_{v(y)}$ satisfies:
\begin{itemize}
\item $\eta_y^\#\hat{F}_{v(y)}(z,\lambda p)=\lambda\eta_y^\#\hat{F}_{v(y)}(z,p)$ for all $\lambda>0$ and $(z,p)\in\eball_{\rho/2}(0)\times\real^{n+1}\setminus \{0\}$;
\item $\mu^{-1}\leq \eta_y^\#\hat{F}_{v(y)}(z,\nu)\leq \mu$ and $|D^{I}_p\eta_y^\#\hat{F}_{v(y)}(z,\nu)|\leq \mu$ for all $(z,\nu)\in\eball_{\rho/2}(0)\times \S^n$ and multi-index $I$ with $|I|\leq 3$;
\item $|D^I_zD^J_p\eta_y^\#\hat{F}_{v(y)}(x,\nu)|\leq\mu_1^{|I|}$ for all $(x,\nu)\in\eball_{\rho/2}(0)\times \S^n$ and multi-indices $I,J$ with $|I|+|J|\leq 3$ and $0<|I|\leq 2$.
\end{itemize}
\end{enumerate}
Now suppose instead $F:TN\setminus\{(x,0):x\in N\}\to \R$. We say that $F$ defines a parametric elliptic functional on $N$ if there exist $\mu,\mu_1>0$ such that for all $x\in N$ and $0<\rho<\inj(x)$ we have $\tilde{F}\in \tilde{\mathcal{I}}(\mu,\mu_1,{\rho})$, where $\tilde{F}(z,p) := F(\exp_x(z),D\exp_x|_z(p))$. We denote by $\mathcal{I}(N,\mu,\mu_1)$ the class of parametric elliptic functionals on $N$ satisfying these conditions.

Finally, {given $F\in \mathcal{I}(N,\mu,\mu_1)$,} whenever $S$ is a manifold of dimension $n$, $\iota:S\to N$ is an immersion of class $C^1$, and $\O$ is an admissible open set, we write (where $\iota_\O=\iota|_{S_\O}$)
\begin{equation*}
\mathcal{F}(\iota;\O):=\int_{S_\O}F(\iota_\O(y),\nu_\O(y))\, d\vol_{S_\O}.
\end{equation*}
\end{definition}

\begin{remark}
The area functional in $N$ corresponds to $F(z,p)=|\Lambda_n p^\perp_\#|$.
\end{remark}

\begin{definition}
Let $U\subset N$ be an open set, $\F$ be a parametric elliptic functional on $U$, $\iota:S\rightarrow U$ be a $C^1$ immersion, and $(\O,\psi)$ be a variation pair.
We define the functional
\begin{equation*}
\jfg(t)= \F(\psi_t;\O)+\vol_g(t).
\end{equation*}
We say that $S$ is $\jfg$-stationary in $\O$ if $\left.\frac{d}{dt}\jfg(t)\right|_{t=0}=0$ for every admissible variation $\psi$ supported in $\O$.
\end{definition}

{Note that the notation $\jfg(t)$, $\vol_g(t)$ both suppress the dependence on the variation pair $(\O,\psi)$.}

\begin{remark}
If $\F$ is the area functional and $\iota:S\rightarrow U$ is an immersed hypersurface of class $C^2$ that is $\jfg$-stationary with respect to every variational pair $(\O,\psi)$ supported in $U$, then it follows from the first variation formula that the mean curvature of $S_\O$ is $\vec{H}(y)=g(\iota(y))\nu_\O(y)$ for all $y\in S_\O$.
\end{remark}

\begin{definition}\label{jfg stationary immersion}
Let $\iota:S\rightarrow U$ be an immersed hypersurface of class $C^2$, $\O\subset U$ an admissible set, and $I$ be a non-negative integer.
Suppose $S$ is $\jfg$-stationary in $\O$.

We say that $S$ has strong index bounded by $I$ in $\O$ with respect to ambient deformations if for all subspaces $P\subset\csvf(\O)$ of dimension $I+1$ there exists $\phi\in P$, $\phi\neq 0$ and for all $\psi:(-\varepsilon,\varepsilon)\times S\rightarrow U$ admissible variations of $S$ induced by $\phi$ we have $\frac{d^2}{dt^2}\big|_{t=0}\jfg(t)\geq 0$.
In this case we write
\begin{equation*}
\indexa(S;\O)\leq I.
\end{equation*}
We define the strong index of $S$ in $\O$ as
\begin{equation*}
\indexa(S;\O)  := \inf\{I\in\integer:\indexa(S;\O)\leq I\}.
\end{equation*}
We say $S$ is strongly stable in $\O$ if $\indexa(S;\O)= 0$.

We say that $\indexa(S;U)\leq I$ if $\indexa(S;\O)\leq I$ for all admissible open sets $\O\subset U$ and we define $\indexa(S;U)$ as above.
\end{definition}

\begin{remark}\label{index lower bound}
According to the above definition, $\indexa(S;\O)\geq I$ is equivalent to the existence of a plane $P\subset\csvf(\O)$ of dimension $I$ such that for all $\phi\in P$ there exists a variation induced by $\phi$ such that $\left.\frac{d^2}{dt^2}\right|_{t=0}\jfg(t) < 0$.
\end{remark}

\begin{remark}
Since one could define a more classical version of Morse index with respect to all possible variations, we make a notational distinction with a subscript.
Nevertheless, for the purpose of our applications the notion of Morse index restricted to ambient deformations will be sufficient.
\end{remark}

Next we prove that immersed hypersurfaces of bounded index cannot be unstable on too many disjoint open subsets.
This is the equivalent of \cite{aiex-mccurdy-minter2024}*{Lemma $2.3$} for $\jfg$-stationary hypersurfaces.

\begin{lemma}\label{lemma:unstable disjoint sets}
Let $U\subset N$ be an open set, $S$ an immersed hypersurface of class $C^2$ in $U$, and $\O_1,\O_2\subset U$ be admissible sets with $\O_1\cap\O_2=\emptyset$.
Suppose $S$ is $\jfg$-stationary in $\O=\O_1\cup\O_2$, then
\begin{equation*}
\indexa(S;\O_1)+\indexa(S;\O_2)\leq \indexa(S;\O)
\end{equation*}
\end{lemma}
\begin{proof}
For each $i=1,2$, let $J_i:=\indexa(S;\O_i)$.
In view of Remark \ref{index lower bound} we may assume that there exists a pair of planes $P_i\subset\csvf(\O_i)$ of dimension $J_i$ such that 
$\left.\frac{d^2}{dt^2}\right|_{t=0}\jfg(\psi^i_t)<0$
for all variation $\psi^i$ induced by $\phi_i'\in P_i$.

We can identify $P_1\oplus P_2$ as a subspace in $\csvf(\O)$.
Given $\phi_1'\oplus\phi_2'\in P_1\oplus P_2$ and $\psi$ a variation induced by $\phi_1'\oplus\phi_2'$ we can use a cut-off function to decompose it into two variations $\psi^i_t$ induced by $\phi_i'$ for each $i=1,2$.
Since $\O_1\cap\O_2=\emptyset$, it follows that $\left.\frac{d^2}{dt^2}\right|_{t=0}\jfg(\psi_t)=\left.\frac{d^2}{dt^2}\right|_{t=0}\jfg(\psi^1_t)+\left.\frac{d^2}{dt^2}\right|_{t=0}\jfg(\psi^2_t)<0$, which proves that $\indexa(S;\O)\geq J_1+J_2$.
\end{proof}

In the following lemma we will prove that if $S$ has bounded index then it is locally strongly stable (cf.~\cite{bellettini-wickramasekera2019:arxiv}*{Section 2.2}). 

\begin{lemma}\label{lemma:locally strongly stable}
Let $n\geq 2$, $I\in \{0,1,2,\dotsc\}$ be a non-negative integer, and $U\subset N$ be an open set. Suppose $S$ is an immersed hypersurface of class $C^2$ in $U$, and that $x\in U$, $0<\rho<\inj(x)$ are such that $\nball_\rho(x)\subset U$.
Suppose $S$ is $\jfg$-stationary in $\nball_\rho(x)$ and $\indexa(S;\nball_\rho(x))\leq I$.
There exists $0<\bar\rho\leq\rho$ such that $S$ is strongly stable on $\nball_{\sigma}(x)$ for all $0<\sigma\leq\bar{\rho}$.
\end{lemma}
\begin{proof}
Suppose this were false, that is, for all $\rho'>0$ we have $\indexa(S;\nball_{\rho'}(x))\geq 1$. First we observe that since $\indexa(S;\nball_{\rho'}(x))$ is non-decreasing in $\rho'$, hence there exists $\rho_1>0$ such that $\indexa(S;\nball_{\rho'}(x))=\indexa(S;\nball_{\rho_1}(x))$ is constant for all $0<\rho'\leq\rho_1$.

It therefore follows from Lemma \ref{lemma:unstable disjoint sets} that $\indexa(S;\nball_{\rho_1}(x)\setminus\overline{\nball_{\rho'}(x)})=0$ for all $0<\rho'<\rho_1$.
We may use a standard cut-off argument to see that $\indexa(S;B^N_{\rho_1}(x)) = 0$, which therefore gives a contradiction and thus concludes the proof. 
\end{proof}

\subsection{Regular and Singular Sets of a Varifold}
\hfill

Henceforth, let $U\subset N$ be an open set and $V\in\ivarifolds_n(U)$ an integral $n$-varifold in $U$.
The following definitions are the same as \cite{bellettini-wickramasekera2019:arxiv}*{Definition 1.3 -- Definition 1.7}.

\begin{definition}[$C^k$ Regular Set]\label{reg set}
Let $k\in \{1,2,\dotsc\}\cup\{\infty\}$. We say $x\in\spt\mass{V}$ is a $C^k$ regular point of $V$ if there exists $r>0$ such that $\spt\mass{V}\cap\nball_r(x)$ is a $C^k$ embedded hypersurface in $\nball_r(x)$.

Write $\reg_k\,V$ for the set of all regular points of $V$ of class $C^k$. We also write $\sing\, V := \spt\|V\|\setminus\reg_2\, V$ for the set of singular points of $V$.
\end{definition}

We note that singular points may correspond to points of lower regularity or even to smooth points where the varifold is locally represented by an immersion. Since we will be interested in PMC hypersurfaces where self-touching points may occur, we will further decompose the singular set.
\begin{definition}[Classical Singularities]\label{classical singularities}
We say that $x\in\sing\,V$ is a classical singularity of $V$ if there exists $r>0$ such that $\spt\mass{V}\cap\nball_r(x)$ can be written as the union of at least three embedded hypersurfaces-with-boundary of class $C^{1,\alpha}$ for some $\alpha\in(0,1]$ with pairwise intersection only on their common $C^{1,\alpha}$ boundary containing $x$ and such that at least one pair of intersections is transverse.

We write $\sing_C\,V$ for the set of all classical singularities.
\end{definition}

\begin{center}
\begin{tikzpicture}[scale=0.8]

\draw (-2.5,3.5) node (v1) {} .. controls (-1.5,3.5) and (-0.5,2.5) .. (-0.5,1.5) node (v3) {};
\draw (0,3.5) node (v2) {} .. controls (1,3.5) and (2,2.5) .. (2,1.5) node (v4) {};
\draw  (v1) edge (v2);
\draw[very thick]  (v3) edge (v4);
\draw (-0.5,1.5) .. controls (0,0.5) and (1,0) .. (1.5,0.5) node (v5) {};
\draw (2,1.5) .. controls (2.5,0.5) and (3.5,0) .. (4,0.5) node (v6) {};
\draw  (v5) edge (v6);
\draw (-0.5,1.5) .. controls (-1.5,1.5) and (-2.5,0.5) .. (-2.5,-0.5) node (v7) {};
\draw[dashed] (2,1.5) .. controls (1,1.5) and (0,0.5) .. (0,-0.5) node (v8) {};
\draw  (v7) edge (v8);

\draw (7,3) node (v9) {} -- (7,-0.5) -- (9,0.25) -- (9,3.5) -- (v9);
\draw (6.5,2) node (v10) {} -- (11,2) -- (9.5,1) -- (5,1) -- (v10);
\draw[very thick] (9,2) -- (7,1);
\node[label={[label distance = -0.2 cm]30:$\textnormal{sing}_C\,V$}] at (2,1.5) {};
\node[label={[label distance = -0.2 cm]85:$\textnormal{sing}_C\,V$}] at (9,2) {};
\end{tikzpicture}
\end{center}

\begin{definition}[Touching Singularities]\label{touching singularities}
We say that $x\in\sing\,V\setminus (\reg_1\,V\cup\sing_C\,V)$ is a touching singularity if  there exist $r>0$ and two embedded hypersurfaces $M_1,M_2$ in $\nball_r(x)$ of class $C^{1,\alpha}$ (for some $\alpha\in(0,1]$) such that $\spt\mass{V}\cap\nball_r(x)=M_1\cup M_2$.
For $r'<r$, the coincidence set of $V$ in $\nball_{r'}(x)$ is defined as $C_{x,r'}=M_1\cap M_2\cap\nball_{r'}(x)$.

We write $\sing_T\,V$ for the set of all touching singularities of $V$.

\end{definition}

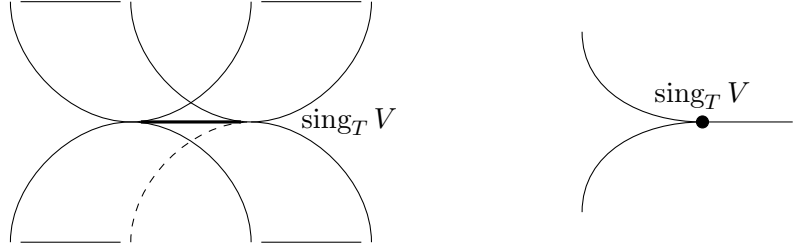
\begin{center}
\begin{tikzpicture}[scale=0.8]

\draw (0,1) node (v3) {} .. controls (0,2) and (1,3) .. (2,3) .. controls (3,3) and (4,2) .. (4,1) node (v5) {};
\draw[dashed] (2,1) node (v4) {} .. controls (2,2) and (3,3) .. (4,3);
\draw (4,3) .. controls (5,3) and (6,2) .. (6,1) node (v6) {};

\draw (0,5) node (v7) {} .. controls (0,4) and (1,3) .. (2,3) node (v1) {} .. controls (3,3) and (4,4) .. (4,5) node (v9) {};
\draw (2,5) node (v8) {} .. controls (2,4) and (3,3) .. (4,3) node (v2) {} .. controls (5,3) and (6,4) .. (6,5) node (v10) {};

\draw[very thick]  (v1) edge (v2);
\draw  (v3) edge (v4);
\draw  (v5) edge (v6);
\draw  (v7) edge (v8);
\draw  (v9) edge (v10);
\node[label={[label distance = 0.4 cm]0:$\textnormal{sing}_T\,V$}] at (4,3) {};

\draw (9.5,4.5) .. controls (9.5,3.5) and (10.5,3) .. (11.5,3) .. controls (12.5,3) and (12,3) .. (13,3);
\node (v11) at (11.5,3) {};

\draw[fill]  (v11) circle (0.1);

\draw (9.5,1.5) .. controls (9.5,2.5) and (10.5,3) .. (11.5,3);
\node[label={[label distance = -0.1 cm]90:$\textnormal{sing}_T\,V$}] at (11.5,3) {};
\end{tikzpicture}
\end{center}

\begin{remark}\label{graph remark}
Observe that if $x\in\sing_T\,V$ and $M_1,M_2$ are as above, then $x\not\in\reg_1\,V$ implies that $x\in M_1\cap M_2$ and $x\not\in\sing_C\,V$ implies $T_xM_1=T_xM_2=T$.
In particular, if we choose $r<\inj(x)$, then for $i=1,2$ we can find $C^{1,\alpha}$ functions $u_i:\eball^n_r(0)\cap T\rightarrow \real$ satisfying $u_1(0)=u_2(0)=0$ and $Du_1(0)=Du_2(0)=0$ such that $M_i=\{\exp_x(u_i(q)\nu)\in\nball_r(x):q\in\eball_r^n(0)\cap T\}$, where $\nu\in T_xN$ is a normal vector to the plane $T$. Furthermore, when $V$ has no classical singularities we may arrange so that $u_1\leq u_2$.
\end{remark}

\begin{definition}[Generalized Regular Set]\label{gen-reg set}
We say that $x\in\spt\mass{V}$ is a generalized regular point if either $x\in\reg_2\,V$ or $x\in\sing_T\,V$ and we may choose $u_1,u_2$ as in Remark \ref{graph remark} to be $C^2$ with $u_1\leq u_2$ (in particular, there are no classical singularities limiting to $x$).

We write $\genreg\,V$ for the set of all generalized regular points of $V$.
\end{definition}

\begin{remark}
Observe that $\genreg\,V$ can be identified with a $C^2$ immersion $\iota:S\rightarrow U$.
\end{remark}

Finally we define the class of varifolds for which our main theorem holds.

\begin{definition}[\cite{bellettini-wickramasekera2019:arxiv}*{Definition 6.2}]\label{admissible varifolds}
Let $U\subset N$ be an open set, $g:U\rightarrow\real$ be a function of class $C^{1,\alpha}$, $I\in \{0,1,2,\dotsc\}$ be a non-negative integer, and $\mu,\mu_1,\Lambda>0$.
We define the class $\sclass_{g,I}(U) \equiv \sclass_{g,\Lambda,I,\mu,\mu_1}(U) \subset \ivarifolds_n(U)$ of integral varifolds $V$ such that for some parametric elliptic functional $F\in\mathcal{I}(N,\mu,\mu_1)$ we have:
\begin{enumerate}
\item [\textnormal{(s1)}] $\mass{\delta V}$ is a Radon measure, $\mass{\delta V}_{\textnormal{sing}}=0$, $|H(V)|\in\lploc{q}(U;\mass{V})$ for some $q>n$, where $H(V)$ is the generalized mean curvature of $V$, and $\mass{V}(U)\leq \Lambda$;
\item [\textnormal{(s2)}] $\sing_C\,V=\emptyset$;
\item [\textnormal{(s3)}] For every $x\in\sing_T\,V$ there exists $r'>0$ such that $\haus^n(C_{x,r'}\cap \{g\neq 0\})=0$;
\item [\textnormal{(s4)}] The embedded hypersurface $\reg_1\,V$ is $\jfg$-stationary;
\item [\textnormal{(s5)}] For every $x\in\sing_T\,V$ and $C^{1,\alpha}$ functions $u_1\leq u_2$ as in Remark \ref{graph remark}, for $i=1,2$ the $\jfg$-stationarity of the embedded hypersurface $\reg_1\,V\cap\graph(u_i)$ (which follows from \textnormal{(s4)}) holds for the orientation which agrees with one of the two possible orientations of $\graph(u_i)$;
\item [\textnormal{(s6)}] The set $\genreg\,V$ is the image of an immersion $\iota:S\rightarrow U$ of class $C^2$ such that for every simply connected open set $\Omega\subset\subset U$ and relatively closed $Z\subset\Omega$ with $\haus^{n-7+\beta}(Z)=0$ for all $\beta\in(0,1)$, for $\O=\Omega\setminus Z$ we have:
\begin{enumerate}
\item [\textnormal{(i)}] $V\restrictv\O = \setv(\iota_O(S_\O),\theta_{\iota_O})$, where $\theta_{\iota_\O}(x) :=|\iota_\O^{-1}(x)|$ whenever $x\in S_\O$ and $\theta_{\iota_{\O}}(x) := 0$ otherwise;
\item [\textnormal{(ii)}] $\iota_{\O}(S_\O)$ is an orientable proper $\jfg$-stationary immersion in $\O$;
\item [\textnormal{(iii)}] $\indexa(S_\O;\O)\leq I$.
\end{enumerate}
\end{enumerate}
\end{definition}

For the rest of this section, we fix $U\subset N$ an open set, $g:U\to \R$ a $C^{1,\alpha}$ function, $I\in \integer_{\geq 0}$, and $\mu,\mu_1,\Lambda>0$. As suggested in our notation, often we omit the dependence of $\sclass_{g,I}$ on $\Lambda,\mu,\mu_1$.

\begin{remark}\label{remark:app-A}
Observe that condition $\textnormal{(s5)}$ is redundant when $g>0$.
If in addition to $g>0$ we assume that at every touching singularity the two graphs given by Remark \ref{graph remark} have constant integer multiplicity, then condition $\textnormal{(s6)(i)}$ is also redundant (see \cite{bellettini-wickramasekera2019:arxiv}*{Page 8, Condition (\textbf{m})}).
We further note that the open set $\O$ in condition $\textnormal{(s6)}$ is simply connected and therefore orientable (see Corollary \ref{k-connected}) so that it is an admissible set as in Definition \ref{admissible sets and variations}.
\end{remark}

\begin{remark}
In general one could define a notion of \emph{weak index} (denoted by $\windexa$) for such hypersurfaces, where in the notion of index we only take into account test functions functions $\phi\in C^1_c(\O)$ which obey $\int_{S_{\O}} \phi(y) g(y)\,d\vol_{S_{\O}}(y)=0$ (balanced functions) as well as variations which are $g$-volume preserving, and replace $\textnormal{(s6)(iii)}$ with weak index to define an alternative class of varifolds. However, when the function $g$ is allowed to change sign it is not always possible to guarantee the existence of $g$-volume preserving variations for every balanced function (see \cite{bellettini-wickramasekera2019:arxiv}*{Remark 1.11}).

If we assume that the function $g$ has a sign (i.e.~$g>0$ or $g<0$ everywhere) then $g$-volume preserving variations always exist (see \cite{barbosa-docarmo-eschenburg1988}*{Lemma 2.2}) so the weak index is not vacuously zero.
In this case, defining the above class with respect to weak index bounds would not change our main result.
Indeed one can easily prove the equivalent version of Lemma \ref{lemma:unstable disjoint sets} and that the strong index and weak index are related by $\windexa\leq\indexa\leq\windexa+1$, from which it readily follows that finite weak index implies locally strongly stable. The proof of Theorem \ref{main theorem n>7} will then follow exactly as is presently below. Since we have that $\indexa+1\leq\windexa+2\leq2(\windexa+1)$, the conclusion of Theorem \ref{main theorem n>7} could then also be stated in terms of weak index.
\end{remark}

\vspace{0.5em}

The following lemma follows directly from Lemma \ref{lemma:locally strongly stable}. Consequently, all of the results for stable prescribed mean curvature hypersurfaces in \cite{bellettini-wickramasekera2019:arxiv} hold at sufficiently small scales.

\begin{lemma}\label{lemma:admissible bdd index implies admissible locally stable}
If $V\in\sclass_{g,I}(U)$ then for every $x\in \spt\mass{V}$ there exists $r = r(x)>0$ such that $V\restrictv\nball_r(x)\in\sclass_{g,0}(\nball_r(x))$.
\end{lemma}

Next we introduce the \emph{stability radius}:

\begin{definition}\label{defn:stab-radius}
Let $V\in \sclass_{g,I}(U)$ and denote by $\iota:S\rightarrow U$ the immersion corresponding to $\genreg\,V$. We define the stability radius $s_V:U\to [0,\infty]$ as:
\begin{equation*}
s_V(x):= \sup\{r\geq 0:\indexa(S_{\O};\O)=0,\text{ for all }\O\subset\nball_r(x)\text{ as in Definition }\ref{admissible varifolds}(s6)\}.
\end{equation*}
\end{definition}

It follows from Lemma \ref{lemma:admissible bdd index implies admissible locally stable} that $s_V(x)>0$ for all $x\in\spt\mass{V}$.
The stability radius is then in fact Lipschitz; the proof of this is exactly the same as in \cite{aiex-mccurdy-minter2024}*{Lemma 2.5}

\begin{lemma}\label{lemma:stability radius continuity}
Let $V\in \sclass_{g,I}(U)$.
Then, either $s_V\equiv \infty$ or for all $x,y\in U$ or we have
$$|s_V(x)-s_V(y)|\leq d^N(x,y).$$
\end{lemma}

Next we define the \emph{regularity scale} for a general integral varifold $V$ (cf.~\cites{aiex-mccurdy-minter2024,naber-valtorta2020}).
\begin{definition}\label{regularity scale}
    Let $U\subset N$ be an open set, $V\in \ivarifolds_n(U)$, $x\in \spt\|V\|$, and $Q\in \{1,2,\dotsc\}$ be a positive integer. 
    First define
		\begin{equation*}
		\begin{aligned}
    \regscale^Q_{0,V}(x):= \sup\{\rho\in (0,\inj(x)): \nball_\rho(x)\subset U\text{ and }& V\restrictv\nball_\rho(x)\textit{ is a union of at most }\\
		                 & \hspace{6em} Q\text{ graphs of $C^2$ functions}\}
		\end{aligned}
		\end{equation*}
    where we set $\sup\emptyset := 0$. 
   
    When $V\restrictv\nball_\rho(x)$ is represented by the sum of $q\leq Q$ graphs of $C^2$ functions $u_1,\dotsc,u_q$, the regularity scale of $V$ at $x\in \spt\|V\|$ is defined to be
    $$\regscale^Q_{V}(x):= \sup\left\{0<\rho\leq \regscale^Q_{0,V}(x): \sup_{\spt\|V\|\cap \nball_\rho(x)}\rho\sum^q_{i=1}|A_{\graph(u_i)}| \leq 1\right\}$$
    where, once again, if this set is empty (i.e. $\regscale^Q_{0,V}(x) = 0$) we set $\regscale^Q_{V}(x):= 0$.
    
    We then write $\badreg_r^Q(V):= \{x\in\spt\|V\|:\regscale^Q_V(x)\leq r\}$ and $\badreg_r(V):=\cup_{Q\geq 1}\badreg_r^Q(V)$.
\end{definition}

\begin{lemma}\label{regularity scale at genreg}
Let $V\in \sclass_{g,I}(U)$.
Then, $x\in\genreg\,V$ if and only if $\regscale^Q_{V}(x)>0$ for some $Q\geq 2$.
\end{lemma}
\begin{proof}
It follows directly from the definition that if $x\in\genreg\,V$ then $\regscale^{Q}_V(x)>0$ for $Q = \max\{2,\Theta_V(x)\}$.

Now suppose $x\in\spt\mass{V}$ is such that $\regscale^Q_V(x)>0$. Then $V$ can be locally written as the graph of at most $Q$ functions of class $C^2$.
In particular $V$ has a unique tangent cone $\mathbf{C}$ at $x$ whose support is the union of at most $Q$ hyperplanes.
Since $V$ is locally strongly stable, it follows from the Minimum Distance Theorem \cite{bellettini-wickramasekera2019:arxiv}*{Theorem 6.3} that $\spt\,\mathbf{C}$ must be a single hyperplane, that is, either $x\in \reg_2\, V$ or some subset of the $Q$ functions intersect tangentially at $x$. It follows from \cite{bellettini-wickramasekera2019:arxiv}*{Theorem 6.4(ii)} that the $Q$ functions only intersect in pairs, hence $x\in \sing_T\,V\cap \genreg\, V$.
\end{proof}

\begin{remark}
In the above we only used that $V$ has no classical singularities to show that a local graphical representation can only intersect tangentially.
\end{remark}

\section{Quantitative Stratification}\label{sec:quantitative}

For the remainder of the paper we fix a manifold $N$ of dimension $n+1$ and a point $0\in N$ such that $\nball_2=\nball_2(0)$ is orientable and for some $K\in (0,1/2)$, we have $\left|\left.\textnormal{sec}\right|_{B^N_2(0)}\right| \leq K$ and $\left.\inj\right|_{B^N_2(0)}\geq K^{-1}$.

In this section we briefly recall the notion of \textit{strata} as well as the \textit{quantitative strata} in codimension one.
For more details see \cite{aiex-mccurdy-minter2024}*{Section 3}. Here, we denote by $\kcone_k\subset\ivarifolds_n(\real^{n+1})$ the set of $k$-symmetric $n$-dimensional cones in $\real^{n+1}$.

\begin{definition}\label{defn:conical}
    Let $\delta>0$, $r\in (0, K^{-1})$, and $k\in \{0,\dotsc,n\}$. We say that $V\in\ivarifolds_n(\nball_2)$ is $(\delta,r,k)$-conical at a point $x\in \spt\|V\|$ if $\nball_r(x)\subset \nball_2$ and there exists a $k$-symmetric cone $\BC\in\kcone_k$ such that
    $$\mathbf{d}\left((\eta_{0,r})_\#\tilde{V}\restrictv \eball_1,\BC\restrictv\eball_1\right)\leq \delta;$$
where $\tilde{V}=(\exp_x^{-1})_\# (V\restrictv\nball_r(x))\in\ivarifolds_n(T_xN)$ and $\mathbf{d}$ is the metric corresponding to the Fr\' echet structure of the varifold topology which induces the same topology.
\end{definition}

\begin{definition}\label{defn:strata}
    Let $\delta>0$, $R\in (0,K^{-1}]$, $r\in (0,R)$, and $V\in \ivarifolds_n(\nball_2)$ with bounded first variation. Then for each $k\in\{0,\dotsc,n\}$, we define the $k^{\text{th}}$ $(\delta,r,R)$-stratification by:
\begin{equation*}
\strata^k_{\delta,r,R}(V):= \{x\in\spt\|V\|: V\text{ is not }(\delta,s,k+1)\text{-conical at }x\text{ for all }s \in [r,R)\}.
\end{equation*}
\end{definition}

The main result of this section is the following recasting of the regularity results in \cite{bellettini-wickramasekera2019:arxiv}.

\begin{theorem}[$\varepsilon$-Regularity Theorem]\label{epsilon regularity}
Let $n\geq 2$, $\Lambda\in (0,\infty)$, $\Gamma\in (0,\infty)$, $\mu,\mu_1>0$, and $d,K\in (0,1/2)$. Suppose $(N,h)$ is a Riemannian manifold of dimension $n+1$ with $0\in N$, $\left|\left.\textnormal{sec}\right|_{B^N_2}\right| \leq K$, $\left.\inj\right|_{B^N_2}\geq K^{-1}$, $g:\nball_2\rightarrow\real$ is a function of class $C^{1,\alpha}$, and $A\subset B^N_2(0)$ is a non-empty compact subset which obeys $d(A,\del B^N_2(0))\geq d$. 
Then there exist constants $\varepsilon_0 = \varepsilon_0(n,\mu,\mu_1,\Lambda,K,d,\Gamma,\alpha)\in (0,1)$ and $Q_0 = Q_0(n,\mu,\mu_1,\Lambda,K,d,\Gamma,\alpha)\in \mathbb{Z}_{\geq 1}$ such that the following holds: if $|g|_{1,\alpha}\leq\Gamma$, $V\in \ivarifolds_n(\nball_2)$, $x\in \spt\|V\|\cap A$, and $\rho\in (0,d]$ satisfy:
    \begin{enumerate}
        \item [\textnormal{(1)}] $V\restrictv B^N_{\rho/2}(x)\in\sclass_{g,0}(B^N_{\rho/2}(x))$;
        \item [\textnormal{(2)}] $V$ is $(\varepsilon_0,\rho/2,n-6)$-conical at $x$;
    \end{enumerate}
    then we have $\regscale^{Q_0}_V(x)\geq \varepsilon_0\rho$.
\end{theorem}
\begin{proof}
Suppose the theorem were false. Then for each $i\in \{1,2,\dotsc\}$ we can find $C^{1,\alpha}$ functions $g_i:\nball_2\rightarrow\real$ with $|g_i|_{1,\alpha}\leq\Gamma$ as well as $V_i\in \ivarifolds_n(\nball_2)$, $x_i\in A$ and $\rho_i\in(0,d]$ such that $V_i$ is stable in $\nball_{\rho_i/2}(x_i)$ and $(1/i,\rho_i/2,n-6)$-conical at $x_i$ but $\regscale^{i}_{V_i}(x_i) < \rho_i/i$.
Denote $\tilde{V}_i=\left(\exp_{x_i}^{-1}\right)_{\#}\left(V_i\restrictv\nball_{\rho_i/2}(x_i)\right)$ and $\tilde{g}_i=\sqrt{|h|}g_i\circ\exp_{x_i}$ so that $\tilde{V}_i\in\sclass_{\tilde{g}_i,0}(\eball_{\rho_i/2})$ is $(1/i,\rho_i/2,n-6)$-conical at $0$ and $\regscale^{i}_{\tilde{V}_i}(0)\leq \rho_i/i$, where $\eball_{\rho_i/2}\subset T_{x_i}N$ with respect to the metric $h(x_i)$ (namely, the metric on $T_{x_i}N$ induced by pulling back $h$ via $\exp_{x_i}$).

There exists a subsequence (which we pass to) such that $\{x_i\}$ converges to some $\tilde{x}\in A$.
We assume that $N$ is isometrically embedded in $\real^L$ for sufficient large $L$, in which case we may identify $T_{x_i}N$ with $T_{\tilde{x}}N$ by a (small) rigid motion of $\real^L$.
Under these identifications we may further identify $\tilde{V}_i$ and $\tilde{g}_i$ accordingly in $T_{\tilde{x}}N$.
We write $\hat{W}_i=\left(\eta_{0,\rho_i/2}\right)_{\#}\tilde{V}_i$ and $\hat{g}_i=\left(\rho_i/2\right)\tilde{g}_i$.
Hence, $\hat{W}_i\in\sclass_{\hat{g}_i,0}(\eball_1)$, with $\eball_1\subset T_{\tilde{x}}N$ and, due to the above identifications, there exists a constant $\delta=\delta(n,K)>0$ such that $\hat{W}_i$ is $(\delta/i,1,n-6)$-conical at $0$ and $\regscale^i_{\hat{W}_i}(0)<\delta/i$.

Given $\beta\in(0,\alpha)$ we may find a subsequence of $\{g_i\}$ that converges to a $C^{1,\alpha}$ function $g:\eball_1\rightarrow\real$ in the $C^{1,\beta}$ topology.
We define $\tilde{g}=\sqrt{|h|}g\circ\exp_{\tilde{x}}$ and $\hat{g}=(\hat{\rho}/2)\tilde{g}$, where $\hat{\rho}\in[0,d]$ is the limit of $\{\rho_i\}$ up to a subsequence.
Therefore, $\hat{g_i}$ must converge to $\hat{g}$ in the $C^{1,\beta}$ topology.

It follows from the Compactness Theorem \cite{bellettini-wickramasekera2019:arxiv}*{Theorem 1.4} that there exists a subsequence of $\{\hat{W}_i\}$ that converges to $\hat{W}\in\sclass_{\hat{g},0}(B_1)$ such that $\Sigma=\spt\mass{\hat{W}}\setminus\genreg\,\hat{W}$ satisfies $\haus^{n-7+\gamma}(\Sigma)=0$ for all $\gamma\in(0,1)$.
Furthermore, since $\hat{W}_i$ is $(\delta/i,1,n-6)$-conical at $0$ we know that $\hat{W}$ is a cone with spine $S(\hat{W})$ satisfying $\dim_\haus(S(\hat{W}))\geq n-6$.
Supposing that $\hat{W}$ is not a hyperplane, then the spine does not have embedded points or touching singularities so it must be contained in $\Sigma$, which contradicts the Hausdorff dimension estimates.
Thus, the support of $\hat{W}$ must be a hyperplane.

Since the generalized mean curvature of $\hat{W}_i$ has bounded $L^q$-norm for some $q>n$, it follows that for $i$ sufficiently large $\hat{W}_i$ satisfy the flatness condition of the Sheeting Theorem \cite{bellettini-wickramasekera2019:arxiv}*{Theorem 6.2} on a sufficiently small ball.
Therefore, for $i$ sufficiently large $\hat{W}_i$ is locally given by the graph of $Q$ functions of class $C^{1,\alpha}$. Finally, from the Higher Regularity Theorem \cite{bellettini-wickramasekera2019:arxiv}*{Theorem 6.4} we get that $\hat{W}_i$ is locally given by the graph of $Q$ functions of class $C^2$ (which can be improved to class $C^{2,\alpha}$ by elliptic regularity since $\hat{g_i}$ are of class $C^{1,\alpha}$ and the parametric elliptic functionals are of class $C^{2,\alpha}$).
This contradicts $\regscale^i_{\hat{W}_i}(0)<\delta/i$ for sufficiently large $i$ and concludes the proof.
\end{proof}

As an immediate corollary we have:

\begin{corollary}\label{corollary epsilon regularity}
Let $n\geq 7$, $\Lambda\in (0,\infty)$, $\Gamma\in (0,\infty)$, $\mu,\mu_1>0$, and $d,K\in (0,1/2)$. Suppose $(N,h)$ is a Riemannian manifold of dimension $n+1$ with $0\in N$, $\left|\left.\textnormal{sec}\right|_{B^N_2}\right| \leq K$, $\left.\inj\right|_{B^N_2}\geq K^{-1}$, $g:\nball_2\rightarrow\real$ is $C^{1,\alpha}$, and $A\subset B^N_2(0)$ is a non-empty compact subset which obeys $d(A,\del B^N_2(0))\geq d$.
Then, there exist constants $\varepsilon_0 = \varepsilon_0(n,\mu,\mu_1,\Lambda,K,d,\Gamma,\alpha)\in (0,1)$ and $Q_0 = Q_0(n,\mu,\mu_1,\Lambda,K,d,\Gamma,\alpha)\in \mathbb{Z}_{\geq 1}$ such that the following holds: if $|g|_{1,\alpha}\leq\Gamma$ and $V\in \ivarifolds_n(\nball_2)$ satisfies $V\restrictv\nball_d(x)\in\sclass_{g,0}(\nball_d(x))$ for all $x\in A$, then
\begin{equation*}
\badreg_{\varepsilon_0 \sigma/2}^{Q_0}(V)\cap A\subset \strata^{n-7}_{\varepsilon_0,\sigma/2,d}(V)\cap A
\end{equation*}
for all $\sigma\in(0,d]$.
\end{corollary}

\section{Singular Set Estimates}\label{sec:main theorem}

In this section we will prove the case $n+1\geq 9$ of Theorem \ref{thm:A}.

The first theorem in this section is the rescaled version of the main measure bound by Naber--Valtorta.
The proof of this version is the same as \cite{aiex-mccurdy-minter2024}*{Corollary 4.2} but the constant will depend on the parametric elliptic functional and the upper bound of the generalized mean curvature, which in our case is given by the $C^0$-norm of the prescribed function.
More precisely, the dependency on the parametric elliptic functional is with respect to the constants $\mu,\mu_1$ as in Definition \ref{elliptic functional}.

\begin{theorem}[\cite{naber-valtorta2020}*{Theorem 1.3 and 1.4}]\label{rescaled quantitative estimates}
Let $n\in\mathbb{Z}_{\geq 2}$, $\Lambda\in (0,\infty)$, $\Gamma\in (0,\infty)$, $\mu,\mu_1>0$, $K\in (0,1/2)$, $I\in \{0,1,2,\dotsc\}$, and $\alpha\in(0,1]$.
Suppose $(N^{n+1},h)$ is a Riemannian manifold with $0\in N$, $\left|\left.\textnormal{sec}\right|_{B^N_2}\right| \leq K$, $\left.\inj\right|_{B^N_2}\geq K^{-1}$, and $g:\nball_2\rightarrow\real$ is a function of class $C^{1,\alpha}$.

Then, there exists a constant $C_\varepsilon = C_\varepsilon(n,\mu,\mu_1,\Lambda,K,\Gamma,\varepsilon)\in (0,\infty)$ such that the following is true: if $|g|_{0}\leq \Gamma$, $R\in (0,1/2]$ and $V\in \sclass_{g,I}(\nball_2)$, then
    $$\H^{n+1}\left(\nball_r(\strata^k_{\varepsilon,r,R}(V))\cap \nball_R(x)\right)\leq C_\varepsilon r^{n+1-k}R^k\ \ \ \ \text{for all }r\in (0,R]\text{ and }x\in\nball_1.$$
\end{theorem}

The following two lemmas are the versions of \cite{aiex-mccurdy-minter2024}*{Lemmas 4.3 \& 4.4} for the current setting.
The proof relies solely on Besicovitch covering theorem and Lemma \ref{lemma:unstable disjoint sets}.
We stress that at this point we do not make use of the Minkowski estimates from Theorem \ref{rescaled quantitative estimates} or $\varepsilon$-regularity.
The proofs are exactly the same as in \cite{aiex-mccurdy-minter2024} using the Besicovitch Theorem \cite{federer1969}*{Theorem 2.8.14} for metric spaces (and thus the constant will also depend on the geometric bound $K$ of $(N,h)$).

\begin{lemma}\label{low stability bound}
Let $n\geq 2$, $\Lambda\in (0,\infty)$, $\Gamma\in (0,\infty)$, $\mu,\mu_1>0$, $K\in (0,1/2)$, $I\in \{1,2,\dotsc\}$, and $\alpha\in(0,1]$.
Suppose $(N^{n+1},h)$ is a Riemannian manifold with $0\in N$, $\left|\left.\textnormal{sec}\right|_{B^N_2}\right| \leq K$, $\left.\inj\right|_{B^N_2}\geq K^{-1}$ and $g:\nball_2\rightarrow\real$ is a function of class $C^{1,\alpha}$.

Then, there exists $C_0 = C_0(n,K)\in (0,\infty)$ such that for any varifold $V\in \sclass_{g,I}(\nball_2)$ we have
    $$\H^{n+1}\left(\nball_r(s_V^{-1}(0,r))\cap \nball_{1/2}\right) \leq C_0 I r^{n+1}\ \ \ \ \text{for all }r\in (0,1/2].$$
\end{lemma}

\begin{lemma}\label{covering lemma}
Let $n\geq 2$, $\Lambda\in (0,\infty)$, $\Gamma\in (0,\infty)$, $\mu,\mu_1>0$, $K\in (0,1/2)$, $I\in \{1,2,\dotsc\}$, and $\alpha\in(0,1]$.
Suppose $(N^{n+1},h)$ is a Riemannian manifold with $0\in N$, $\left|\left.\textnormal{sec}\right|_{B^N_2}\right| \leq K$, and $g:\nball_2\rightarrow\real$ is a function of class $C^{1,\alpha}$.

Then, there exists a constant $C_0 = C_0(n,K)\in (0,\infty)$ such that the following holds: if $V\in\sclass_{g,I}(\nball_2)$, then for any $0<a<b<\infty$, $\gamma\in (0,1)$, and subset $A\subset s_V^{-1}([a,b])\cap \nball_1$, there exists a finite set $B\subset A$ such that
    $$A\subset\bigcup_{y\in B}\nball_{\gamma s_V(y)}(y)$$
    and moreover we have the size bound
    $$|B|\leq C_0 I\left(\frac{b}{a}\right)^{n+1}\left(1+\gamma^{-1}\right)^{n+1}.$$
\end{lemma}

The proof of the main theorem now follows the same ideas as in \cite{aiex-mccurdy-minter2024} with the corresponding auxiliary results for the prescribed mean curvature setting.

\begin{theorem}\label{main theorem n>7}
Let $n\geq 8$, $\Lambda\in (0,\infty)$, $\Gamma\in (0,\infty)$, $\mu,\mu_1>0$, $K\in (0,1/2)$, $I\in \{0,1,2,\dotsc\}$, and $\alpha\in(0,1]$.
Suppose $(N^{n+1},h)$ is a Riemannian manifold with $0\in N$, $\left|\left.\textnormal{sec}\right|_{B^N_2}\right| \leq K$, $\left.\inj\right|_{B^N_2}\geq K^{-1}$ and $g:\nball_2\rightarrow\real$ is a function of class $C^{1,\alpha}$. If $|g|_{1,\alpha}\leq \Gamma$ and $V\in\sclass_{g,\Lambda,I,\mu,\mu_1}(\nball_2)$ then $\Sigma=\spt\mass{V}\setminus\genreg\,V$ is countably $(n-7)$-rectifiable and we have, for any $0<r\leq 1/2$:
\begin{equation*}
\begin{aligned}
\haus^{n+1}\left(\nball_{r/8}\left(\Sigma\right)\cap\nball_{1/2}\right)\leq C_0(1+I)r^8;\\
\mass{V}\left(\nball_{r/8}\left(\Sigma\right)\cap\nball_{1/2}\right)\leq C_0(1+I)r^7.
\end{aligned}
\end{equation*}
In particular, $\mathcal{M}^{*n-7}(\Sigma)\leq C_0(1+I)$. Here, $C_0=C_0(n,\mu,\mu_1,\Lambda,K,\Gamma,\alpha)\in(0,\infty)$.
\end{theorem}
\begin{proof}
Fix $r\in (0,1/2]$.
We observe that $\nball_{r/8}(\Sigma)\cap \eball_{1/2}\subset \nball_{r/8}(\Sigma\cap \nball_1)$. 
Write
\begin{equation*}
\Sigma = \left(s_V^{-1}(0,r)\cup s_V^{-1}([r,1])\cup s_V^{-1}((1,\infty])\right)\cap \Sigma.
\end{equation*}

Firstly it follows from Lemma \ref{low stability bound} that 
\begin{equation*}
\H^{n+1}\left(\nball_{r/8}(s_V^{-1}(0,r)\cap\Sigma)\cap \nball_{1/2}\right) \leq C_0Ir^{n+1}
\end{equation*}
where $C_0 = C_0(n,K)\in (0,\infty)$.

Secondly, it follows from Lemma \ref{regularity scale at genreg} that if $x\in\Sigma$ then $\regscale^Q_V(x)=0$ for all $Q>0$, that is, $\Sigma\subset \badreg^Q_\sigma(V)$ for all $\sigma>0$ and $Q>0$.
Take $A=\overline{s_V^{-1}((1,\infty])\cap\nball_1}$ in Corollary \ref{corollary epsilon regularity}, which gives
\begin{equation*}
s_V^{-1}((1,\infty])\cap\Sigma\cap\nball_1\subset\strata^{n-7}_{\varepsilon_0,r,1/2}(V).
\end{equation*}
Hence,
\begin{equation*}
\H^{n+1}\left(\nball_{r/8}(s_V^{-1}((1,\infty])\cap\Sigma)\cap \nball_{1/2}\right)\leq C_{\varepsilon_0}(1/2)^{n-7}r^8,
\end{equation*}
where $C_{\varepsilon_0}$ is given by Theorem \ref{rescaled quantitative estimates}.

Finally, set $k_0:= \min\{k\in\mathbb{Z}_{\geq 1}:\tfrac{1}{2}2^{-k}\leq r\}$. 
Then for $j=0,1,\dotsc,k_0$, set
\begin{equation*}
A_j:= s_V^{-1}([2^{-j-1},2^{-j}])\cap\Sigma\cap \nball_1.
\end{equation*}
Now, applying Lemma \ref{covering lemma} with $a = 2^{-j-1}$, $b=2^{-j}$, $\gamma = 1/8$, we obtain for each $j=0,1,\dotsc,k_0$ a set $B_j\subset A_j$ with $|B_j|\leq 18^{n+1}C_1I$, where $C_1 = C_1(n,K)\in (0,\infty)$, and
\begin{equation*}
A_j\subset \bigcup_{y\in B_j}\nball_{s_V(y)/8}(y).
\end{equation*}
Observe that, by continuity of the stability radius (Lemma \ref{lemma:stability radius continuity}), we know that $s_V\geq 2^{-j-1}$ on $\overline{A}_j$, and thus using Lemma \ref{regularity scale at genreg} and Corollary \ref{corollary epsilon regularity} we have that
\begin{equation*}
A_j\subset \strata^{n-7}_{\varepsilon_0,r/8,2^{-j-2}}(V).
\end{equation*}
In particular,
\begin{equation*}
\nball_{r/8}(A_j)\subset\bigcup_{y\in B_j}\nball_{r/8}\left(\strata^{n-7}_{\varepsilon_0,r/8, 2^{-j-2}}(V)\right)\cap \nball_{s_V(y)/8 + r/8}(y)
\end{equation*}
and as $s_V(y)\leq 2^{-j}$ in $A_j$ and since $r< 2^{-k_0} \leq 2^{-j}$,
\begin{equation*}
\nball_{r/8}(A_j)\subset\bigcup_{y\in B_j}\nball_{r/8}\left(\strata^{n-7}_{\varepsilon_0,r/8, 2^{-j-2}}(V)\right)\cap \nball_{2^{-j-2}}(y).
\end{equation*}
Hence by Theorem \ref{rescaled quantitative estimates} we can find a constant $C^\prime$ such that for each $y\in B_j$ we have
\begin{equation*}
\H^{n+1}\left(\nball_{r/8}\left(\strata^{n-7}_{\varepsilon_0,r/8,2^{-j-2}}(V)\right)\cap \nball_{2^{-j-2}}(y)\right) \leq C^\prime\cdot\left(r/8\right)^8\cdot (2^{-j-2})^{n-7}.
\end{equation*}
Hence, combining the above,
\begin{equation*}
\H^{n+1}(\nball_{r/8}(A_j)) \leq |B_j|\cdot C^\prime\cdot (r/8)^8 \cdot (2^{-j-2})^{n-7} \leq C_* Ir^8\cdot \left(2^{-n+7}\right)^{j}
\end{equation*}
and this is true for each $j=0,1,\dotsc,k_0$; here $C_* = C_*(n,\mu,\mu_1,\Lambda,K,\Gamma,\alpha)$. Therefore,
\begin{equation*}
\H^{n+1}\left(\nball_{r/8}(A_0\cup A_1\cup\cdots\cup A_{k_0})\right) \leq C_*Ir^8\cdot\sum^{k_0}_{j=0}\left(\frac{1}{2^{n-7}}\right)^{j}.
\end{equation*}
Since $n\geq 8$, the sum on the right-hand side is bounded above by $\sum^\infty_{j=0} \left(\frac{1}{2^{n-7}}\right)^{j}<\infty$, which only depends on $n$; this completes the bound of the second term. Thus combining all the above we have
\begin{equation*}
\H^{n+1}\left(\nball_{r/8}(\Sigma)\cap \eball_{1/2}\right) \leq C_0I r^{n+1} + C^*_1Ir^8 + C_2 r^{8} \leq \tilde{C}(1+I)r^8
\end{equation*}
for some constants $C_0,C^*_1,C_2,\tilde{C}$ only depending on $n$, $\mu$, $\mu_1$, $\Lambda$, $K$, $\Gamma$ and $\alpha$; this therefore completes the proof of the first claimed bound in Theorem \ref{main theorem n>7}.

To prove the second inequality in Theorem \ref{main theorem n>7}, we shall use the first inequality of Theorem \ref{main theorem n>7} to prove a packing estimate. Indeed, we claim the following: for any $r\in (0,1/2]$, we can find a covering of $\Sigma\cap \nball_{1/4}$ by balls $\{\nball_r(x_i)\}_{i=1}^L$ with $x_i\in\Sigma\cap B^N_{1/4}$ and $L\leq C_3(1+I)r^{-(n-7)}$, where $C_3 = C_3(n,\mu,\mu_1,\Lambda,K,\Gamma,\alpha)$. Indeed, to see this simply choose $x_i\in \Sigma\cap \nball_{1/4}$ a maximal collection of points such that $\{\nball_{r/2}(x_i)\}_{i}$ is a pairwise disjoint collection. 
Then, by construction we have $\Sigma\cap \nball_{1/4} \subset \bigcup_{i=1}^L \nball_r(x_i)$, and
\begin{align*}
L r^{n-7} = \sum_{i=1}^L r^{n-7} & \leq \hat{C}_1\omega_{n+1}^{-1}2^{n+1}\cdot r^{-8}\sum_{i=1}^L \H^{n+1}(\nball_{r/2}(x_i))\\
&\leq \hat{C}_1\omega_{n+1}^{-1}2^{n+1}\cdot r^{-8} \H^{n+1}(\nball_{r/2}(\Sigma\cap \nball_{1/4}))\\
& \leq \hat{C}_1\omega_{n+1}^{-1}2^{n+1}\cdot r^{-8}\cdot C(1+I)r^{8}\\
& \equiv C_3(1+I)
\end{align*}
where $C_3 = \omega^{-1}_{n+1}2^{n+1}\hat{C}_1C$, where $\hat{C}_1=\hat{C}_1(n,K)$ only depends on $n$ and the geometry of $N$ and $C = C(n,\mu,\mu_1,\Lambda,K,\Gamma,\alpha)$ is the constant from the first inequality we have already established; this proves the packing estimate claimed. Thus, using this cover we now have
\begin{align*}
\|V\|(\nball_{r/8}(\Sigma)\cap \nball_{1/8}) & \leq \|V\|(\nball_{r}(\Sigma\cap \nball_{1/4}))\\
&\leq \|V\|\left(\bigcup_{i=1}^L\nball_{2r}(x_i)\right) \leq \sum_{i=1}^L\|V\|(\nball_{2r}(x_i))\\
& \leq \sum_{i=1}^L \hat{C}_2(2r)^n\|V\|(\nball_1(x_i))\\
& \leq \hat{C}_22^n\Lambda \cdot r^n L \leq 2^n\Lambda \hat{C}_2C_3(1+I) r^7
\end{align*}
where $\hat{C}_2=\hat{C_2}(n,\Lambda)$; in the fourth inequality here we have used the monotonicity formula for varifolds with locally bounded first variation \cite{simon1983:gmtbook}*{\textsection 17}. This therefore completes the proof of the second estimate.
\end{proof}

\section{Ambient Dimension at Most $8$}\label{case n+1=8}

The remaining case of Theorem \ref{thm:A} is when $n+1=8$. Here, the singular set will consist of isolated points and thus the size of the singular set is invariant under local scaling. The previous proof given when $n+1\geq 9$ therefore will not work. In fact, the Minkowski estimate is equivalent to estimating the (finite) number of singular points, that is, estimating $\haus^0$. The mean curvature zero case of Theorem \ref{thm:A} (when $n+1=8$) was proved by Song in \cite{song.a2023}, and we will follow those ideas to prove Theorem \ref{thm:A} when $n+1=8$ in our setting as well. Indeed, the core of Song's proof is purely combinatorial and does not rely on the specific structure of the variational object other than through various uniform estimates, which we have in our setting. However, some minor modifications are needed to accommodate our PMC hypersurfaces, which we detail here.

Once we have these modifications, Theorem \ref{thm:A-0} also follows via the same arguments in \cite{song.a2023}.

Therefore, we will only highlight the small changes needed to verify that Song's arguments still hold in the present setting, without repeating proofs unnecessarily. Essentially, one is able to replace every instance in Song's work where the Schoen--Simon \cite{schoen-simon1981} sheeting theorem is used with the corresponding sheeting theorem of Bellettini--Wickramasekera \cite{bellettini-wickramasekera2019:arxiv}. We also include a proof of \cite{song.a2023}*{Page 2085, (7)} in a more general form in Appendix \ref{app:B}.

For these reasons, throughout this section we will focus on the case $n+1=8$.

We begin by observing that the notion of stability radius $\textbf{r}_{\text{stab}}$ defined in \cite{song.a2023}*{Definition 2.1}, with $\lambda=2$ and $\bar{r}>0$, corresponds to our notion of stability radius via $\textbf{r}_{\text{stab}}(x)=\min\{\frac{1}{2}s_V(x),\bar{r}\}$. A key part of Song's argument is the definition of an \emph{almost conical region}, in which the hypersurfaces look like \emph{minimal} hypercones. Minimality is important as one uses the Frankel property for the links of minimal cones. In the PMC setting, we need to ensure that $\bar{r}$ is chosen small enough, depending on the upper bound $\Gamma$ on $g$ as well as on $(N,h)$, to ensure this still holds. We will therefore use the notation $\textbf{r}_{\text{stab}}$ in this section to make the distinction clear.

This illustrates the primary modification needed when using Song's arguments in the PMC setting: one needs to allow the various quantities to depend on the upper bound on $\|g\|_{C^{1,\alpha}}$, or otherwise assume in various places that not only is the metric close to Euclidean, but also that $\|g\|_{C^{1,\alpha}}$ is small, so that in contradiction/compactness arguments the limit of the sequence is a \emph{minimal} cone. This can, of course, be achieved by working on sufficiently small balls (i.e.~$\bar{r}$ small) or by rescaling the metric.

Another place one might be concerned with is the presence of touching points in $\genreg\,V$, and so the `smooth' part of $V$ is not embedded. However, touching singularities where the order is preserved is sufficient for the arguments, as then limiting minimal objects will coincide. As our sheeting theorems allow for multiplicity, one does not encounter difficulties. On this note, whenever $y\in\genreg\,V$ is a point of self-intersection we write $A_V(y)$ to be the second fundamental form of $V$ in the sense of Definition \ref{regularity scale}.

To illustrate these points, we first prove a version of the reverse curvature bound similar to \cite{song.a2023}.

\begin{lemma}\label{reverse curvature estimate}
Let $\mu,\mu_1,\Lambda,\Gamma\in(0,\infty)$, $K\in(0,1/2)$, and $\alpha\in (0,1]$. 
Suppose $(N,h)$ is a Riemannian manifold of dimension $8$ with $0\in N$, $|\textnormal{sec}|_{B^N_2}| \leq K$ and $g:\nball_2\rightarrow\real$ is a $C^{1,\alpha}$ function. Then, there exist constants $\varepsilon_0 = \varepsilon_0(\mu,\mu_1,\Lambda,K,\Gamma,\alpha)\in(0,1)$ and $\bar{r} = \bar{r}(\mu,\mu_1,\Lambda,K,\Gamma,\alpha)\in (0,1)$ with the following property.
If $V\in\ivarifolds_7(N)$, $r\in (0,\bar{r})$, and $x\in N$ satisfy:
\begin{enumerate}
\item[\textnormal{(a)}] $V$ satisfies Definition \ref{admissible varifolds}\textnormal{(s1)-(s5)} with respect to $g$, where $|g|_{1,\alpha}\leq \Gamma$; and
\item[\textnormal{(b)}] $2r\leq\inj(x)$ and $V$ is strongly unstable in $B^N_{2r}(x)$;
\end{enumerate}
then either
\begin{enumerate}
\item [\textnormal{(i)}] $(\supp\mass{V}\setminus\genreg\,V)\cap\nball_{2r}(x)\neq\emptyset$; or
\item [\textnormal{(ii)}] there exists $y'\in\nball_{2r}(x)$ such that $r|A_V(y')|>\varepsilon_0$.
\end{enumerate}
\end{lemma}
\begin{proof}
Note that if the metric is sufficiently close to the Euclidean metric on a normal neighbourhood, and the second fundamental form is arbitrarily small after scaling, then the varifold must be stable. The result therefore follows by a simple contradiction argument using the convergence to a (minimal) cone.
\end{proof}

One then defines the set of cones of interest, $\mathcal{G}_{\beta_0}$, analogously to \cite[Page 2086]{song.a2023}, using again \cite[Lemma 22]{song.a2023} to find the corresponding value of $\beta_0$. Write $\inj(N) := \inf\{\inj(x):x\in N\}$ and, for $p\in N$ and radii $0<s<r<\inj(p)$, $A(p;s,r):= B^N_r(p)\setminus \overline{B^N_s(p)}$ for the (open) annulus centered at $p$ with inner radius $s$ and outer radius $r$. One then defines the notion of a \emph{pointed} $\delta$-conical annulus as in \cite[Page 2087]{song.a2023}, again using $\mathcal{G}_{\beta_0}$ as above consisting of \emph{minimal} cones. 

As in \cite[Page 2087 (9)]{song.a2023}, one can then compare the stability radius to the outer radius of a sufficiently small pointed $\delta$-conical annulus as follows: there exists $\delta_0 = \delta_0(\mu,\mu_1,\Lambda,\Gamma, K,\alpha)\in (0,1)$ such that if $V\in\ivarifolds_7(N)$, $0<s<\tfrac{t}{2}<\tfrac{1}{2}\bar{r}$, and $p\in N$ satisfy:
\begin{enumerate}
\item [\textnormal{(a)}] $V$ satisfies Definition \ref{admissible varifolds}\textnormal{(s1)-(s5)} with respect to $g$; and 
\item [\textnormal{(b)}] $A(p;s,t)$ is a pointed $\delta_0$-conical annulus,
\end{enumerate}
then for all $x\in N$ such that $\nball_{\frac{1}{2}\textbf{r}_{\text{stab}}(x)(x)}\cap\partial\nball_t(p)\neq\emptyset$ we have
\begin{equation*}
C_0t\leq\frac{1}{2}\textbf{r}_{\text{stab}}(x)\leq 2\left(\frac{t}{2}+s\right).
\end{equation*}
The proof follows exactly the same argument, applying instead Lemma \ref{reverse curvature estimate} after scaling the metric by $t^{-2}$.

The notion of a $(\delta,K)$-telescope is exactly as in \cite[Section 2.2.4 \& Section 5.1]{song.a2023}, as well as the sets $\mathcal{A}_{\delta}$ and $\mathcal{A}^{\text{bis}}_{\delta}$, for suitable constants $K_0$ therein (in \cite{song.a2023} this constant is denoted $K$, but we denote it $K_0$ to avoid repetition with our use of $K$ as a bound on the sectional curvature). One deduces from the above stability radius comparison with the radii of pointed $\delta_0$-telescopes that \cite[Lemmas 5, 6, \& 23]{song.a2023} all still hold in this setting.

The main theorem needed for the combinatorial arguments (in \cite[Section 2.5]{song.a2023}) now follows, namely:
\begin{theorem}[\cite{song.a2023}*{Theorem 8/24}]\label{thm:song-combinatorial}
Let $\mu,\mu_1, \Lambda \in (0,\infty)$ and $\alpha\in(0,1]$.
Then there exist $\beta_1>1$, $\nu>0$, $\bar{R}>1000$, and $H_1>0$ depending on $\mu,\mu_1, \Lambda, \delta_0$, and $K_0$ such that the following is true.
Let $\bar{K}=60K_0^2$, $h$ be a metric $\nu$-close to the Euclidean metric in the $C^5$-topology on $\eball(0,\bar{K})\subset\real^8$, $g$ be a function of class $C^{1,\alpha}$ and
\begin{equation*}
(\Sigma,\partial\Sigma)\subset(\eball(0,\bar{K}),\partial\eball(0,\bar{K}))
\end{equation*}
be a compact PMC hypersurface (that is, $\setv(\Sigma,1)$ satisfies Definition \ref{admissible varifolds}\textnormal{(s1)-(s6)(ii)}) and such that
\begin{itemize}
\item $\haus^7(\Sigma)\leq \frac{1}{2}\Lambda \bar{K}^7$;
\item $\textbf{r}_{\textnormal{stab}}(0)<\bar{K}^{-1}$;
\item there is $y'\in\eball(0,14)$ with $\textbf{r}_{\textnormal{stab}}(y')=1$;
\item $\Theta_h(0,20K_0)-\Theta_h(0\frac{1}{3K_0})\leq \beta_1^{-1}$;
\item $|g|_{1,\alpha}\leq H_1$.
\end{itemize}
Then one of the following holds:
\begin{enumerate}[\textnormal{(a)}]
\item there exist $z',z''\in\eball(0,\bar{K}/2)$ such that $\eball(z',6\textbf{r}_{\textnormal{stab}}(z'))\cap\eball(z'',6\textbf{r}_{\textnormal{stab}}(z'')) = \emptyset$ and $\textbf{r}_{\textnormal{stab}}(z')$, $\textbf{r}_{\textnormal{stab}}(z'')\in[2^{-(R+1)},2^{-R}) \text{ for some }R\in(0,\bar{R}]$;
\item [\textnormal{(b)}] $\Sigma\cap A(0;\frac{1}{2K_0},14K_0)$ is $\delta_0$-close to a cone $C\in\mathcal{G}_{\beta_0}$.
\end{enumerate}
\end{theorem}

Indeed, arguing by contradiction, we use the corresponding Sheeting Theorem \cite[Theorem 6.2]{bellettini-wickramasekera2019:arxiv} and Higher Regularity Theorem \cite[Theorem 4 \& Theorem 6.4]{bellettini-wickramasekera2019:arxiv} of Bellettini--Wickramasekera in place of those of Schoen--Simon \cite{schoen-simon1981} to obtain in the limit a minimal cone. Then the proof follows the same reasoning as in \cite{song.a2023}*{Theorem 8}. We also need that the density ratio $\Theta_h(0,\cdot)$ of $\Sigma$ (appropriately defined by PMC hypersurfaces in a Riemannian manifold) is approximately monotone increasing and converges to the standard density ratio for minimal hypersurfaces as both the metric converges to the Euclidean metric and the mean curvature converges to $0$ (cf.~\cite[Theorem 5.1]{allard1972}).

We can use Theorem \ref{thm:song-combinatorial} by rescaling the metric on $N$ by a constant that depends only on $h$ and the upper bound $\Gamma$ for the prescribed function $g$. At this point, Song's proof in \cite[Corollary 9, Section 2.4, \& Section 2.5]{song.a2023} is purely combinatorial, relying only on the statements above. Those arguments therefore allow us to conclude Theorem \ref{thm:A} in the remaining case $n+1=8$:

\begin{theorem}\label{main theorem n=7}
Let $I\in{0,1,\ldots}$, $\mu,\mu_1, \Lambda, \Gamma \in (0,\infty)$, $K\in(0,1/2)$, and $\alpha\in(0,1]$. Let $(N,h)$ be a Riemannian manifold of dimension $8$ with $0\in N$, $|\textnormal{sec}|_{B^N_2}| \leq K$, $\left.\inj\right|_{B^N_2}\geq K^{-1}$ and $g:\nball_2\rightarrow\real$ a function of class $C^{1,\alpha}$ with $|g|_{1,\alpha}\leq \Gamma$.
If $V\in\sclass_{g,\Lambda,I,\mu,\mu_1}(\nball_2)$, then $\Sigma=\spt\mass{V}\setminus\genreg\,V$ is a locally finite set in $B^N_2$, and:
\begin{equation*}
\haus^{0}\left(\Sigma\cap B^N_{1/2}\right)\leq C_0(1+I),
\end{equation*}
where $C_0=C_0(\mu,\mu_1,\Lambda,K,\Gamma, \alpha)\in(0,\infty)$.
\end{theorem}

Given the above modifications to the arguments in \cite{song.a2023}, one can deduce Theorem \ref{thm:A-0} in an entirely analogous manner. This completes the proof of our main results.

\appendix

\section{$k$-connected Sets}

In this section we include a simple known topological fact about $k$-connected open sets of a manifold with a set of large Hausdorff codimension removed. This is used in Remark \ref{remark:app-A}. Since we were unable to find a reference, we also include a proof.

\begin{lemma}\label{lemma a1}
Let $n$, $l$, and $k$ be positive integers with $l\leq k\leq n$. Let $S\subset\real^{n+1}$ be a closed set with $\haus^{n-k}(S)=0$ and $\Sigma$ a compact $(l+1)$-dimensional $C^1$ manifold with $\partial\Sigma\neq\emptyset $ and $f:\Sigma\rightarrow\real^{n+1}$ a continuous map with $f(\partial\Sigma)\cap S=\emptyset$.
Then for all $\varepsilon>0$ there exists a continuous map $\tilde{f}:\Sigma\rightarrow\real^{n+1}$ such that:
\begin{enumerate}
    \item [\textnormal{(1)}] $\tilde{f}(\partial\Sigma)=f(\partial\Sigma)$;
    \item [\textnormal{(2)}] $\sup_{x\in\Sigma}|f(x)-\tilde{f}(x)|<\varepsilon$;
    \item [\textnormal{(3)}] $\tilde{f}(\Sigma)\cap S=\emptyset$.
\end{enumerate}
\end{lemma}
\begin{proof}
Firstly, it follows from Whitney's Approximation Theorem (see \cite{hirsch1994}*{\textsection 5.1, Lemma 1.5}) that there exists a continuous map $f':\Sigma\rightarrow\real^{n+1}$ such that $f'(\partial\Sigma)=f(\partial\Sigma)$, $f'$ is $C^1$ on $\Sigma\setminus\partial\Sigma$ and $f'$ is arbitrarily close to $f$ in $C^0$.

Now choose $\delta>0$ sufficiently small so that $\eball_{4\delta}(f'(\partial\Sigma))\cap S=\emptyset$. Choose also a $C^1$ cut-off function $\phi:\real^{n+1}\rightarrow[0,1]$ such that $\phi(y)=0$ on $\eball_{\delta}(f'(\partial\Sigma))$ and $\phi(y)=1$ on $\real^{n+1}\setminus\eball_{2\delta}(f'(\partial\Sigma))$.

Now, define $F:\Sigma\times\real^{n+1}\rightarrow\real^{n+1}$ as $F(x,y)=f'(x)-y$.
It follows that $F$ is $C^1$ and therefore $\haus^{n-k+l+1}(F(\Sigma\times S))=0$.
In particular, $\real^{n+1}\setminus F(\Sigma\times S)$ is dense.
We may therefore pick $\tilde{z}\in\real^{n+1}\setminus F(\Sigma\times S)$ arbitrarily close to $0$ and define $\tilde{f}(x):=f'(x)-\phi(f'(x))\tilde{z}$. We choose $\tilde{z}$ such that $|\tilde{z}|<2\delta$.

We claim that $\tilde{f}(\Sigma)\cap S=\emptyset$.
In fact, since $|\tilde{z}|<2\delta$, if $f'(x)\in\eball_{2\delta}(f'(\partial\Sigma))$ then $\tilde{f}(x)\in\eball_{4\delta}(f'(\partial\Sigma))$ and thus $\tilde{f}(x)\not\in S$. Therefore, if there were $s\in S$ such that $\tilde{f}(x) = s$, then we must have $f'(x)\in \real^{n+1}\setminus B_{2\delta}(f'(\partial\Sigma))$, so that $\tilde{f}(x) = f'(x) - \tilde{z}$. But this then implies $\tilde{z} = f'(x) - s$, which contradicts that $\tilde{z}\in \real^{n+1}\setminus F(\Sigma\times S)$. Thus $\tilde{f}(\Sigma)\cap S = \emptyset$, completing the proof.
\end{proof}

Recall that a path connected topological space $X$ is \emph{$k$-connected} if the homotopy groups $\pi_l(X)$ are trivial for all $1\leq l\leq k$.

\begin{lemma}\label{lemma a2}
Let $k\leq n$ be positive integers, $U\subset\real^{n+1}$ a path connected open set, and $S\subset U$ a closed set with $\haus^{n-k}(S)=0$.
Suppose $U$ is $k$-connected. Then $U\setminus S$ is $k$-connected.
\end{lemma}
\begin{proof}
Fix $1\leq l\leq k$ and let $f:\S^l\rightarrow U\setminus S$ be a continuous map.
Since $\pi_l(U)$ is trivial, $f$ is null-homotopic in $U$, i.e.~there exists a continuous extension $f':\overline{\eball}^{l+1}\rightarrow U$ of $f$.
It follows from Lemma \ref{lemma a1} that $f'$ can be approximated by a ($C^1$) map $\tilde{f}:\overline{\eball}^{l+1}\rightarrow U\setminus S$ with $\tilde{f}\equiv f'\equiv f$ on $\partial B^{l+1}$.
In particular, $f$ is null-homotopic in $U\setminus S$.
\end{proof}

\begin{remark}
Under the same hypotheses and with some technical modifications of the above argument, one can further show that for a triple $(U,A,B)$ with $B\subset A\subset U$ and $S\subset B$ the relative homotopy groups $\pi_l(U,A,B)$ and $\pi_l(U\setminus S,A\setminus S, B\setminus S)$ are in fact isomorphic.
\end{remark}

\begin{lemma}\label{lemma a3}
Let $n$, $l$ and $k$ be positive integers with $l\leq k\leq n$. Let $(N,h)$ be a $(n+1)$-dimensional smooth Riemannian manifold and $S\subset N$ a closed set with $\haus^{n-k}(S)=0$. Suppose $\Sigma$ is a compact $(l+1)$-dimensional $C^1$ manifold with non-empty boundary $\partial\Sigma$ and $f:\Sigma\rightarrow N$ a continuous map with $f(\partial\Sigma)\cap S=\emptyset$.
Then for all $\varepsilon>0$, there exists a continous map $\tilde{f}:\Sigma\rightarrow N$ such that
\begin{enumerate}
    \item [\textnormal{(1)}] $\tilde{f}(\partial\Sigma)=f(\partial\Sigma)$;
    \item [\textnormal{(2)}] $\sup_{x\in\Sigma}d^N(f(x),\tilde{f}(x))<\varepsilon$; 
    \item [\textnormal{(3)}] $\tilde{f}(\Sigma)\cap S=\emptyset$.
\end{enumerate}
\end{lemma}
\begin{proof}
This follows from Lemma \ref{lemma a1} and a globalization argument, cf.~\cite{hirsch1994}*{\textsection 2.2, Theorem 2.2}.
\end{proof}

We can now prove analogously to Lemma \ref{lemma a2}:

\begin{corollary}\label{k-connected}
Let $k\leq n$ be positive integers. Suppose $N$ be an $(n+1)$-dimensional smooth manifold, $U\subset N$ a path connected open subset, and $S\subset U$ a closed set with $\haus^{n-k}(S)=0$. Then if $U$ is $k$-connected, then $U\setminus S$ is $k$-connected.
\end{corollary}

\section{An elementary covering lemma}\label{app:B}

The following lemma is a result regarding the maximum number of disjoint open balls which may intersect a given open ball in certain Riemannian manifolds. As mentioned at the start of Section \ref{case n+1=8}, this builds on \cite{song.a2023}*{Page 2085, (7)} and a simple version is needed in Section \ref{case n+1=8}. The hypotheses are not optimal and the result might hold more generally.

\begin{lemma}
Let $n\in\integer_{\geq 2}$, $\kappa,V\in (0,\infty)$, and $L\in(0,1)$.
Then, there exists a constant $C = C(n,\kappa,V,L)\in (0,\infty)$ with the following property.

Suppose $(N,h)$ is a complete Riemannian manifold of dimension $n+1$ with $|\textup{Ric}_N|\leq \kappa$, $\inj_N>\kappa^{-1}$ and $\textup{vol}(N)\leq V$. Suppose $f:N\rightarrow\real_{>0}$ is a (positive) Lipschitz function with $\textup{Lip}(f)\leq L$. Then for any $y\in N$, we have the following: if $\{B^N_{f(x_i)}(x_i)\}_{i=1}^l$ are pairwise disjoint balls ($x_i \in N$) with
$$B^N_{f(x_i)}(x_i)\cap B^N_{f(y)}(y) \neq \emptyset \qquad \text{for all $i$,}$$
then $l\leq C$.
\end{lemma}
\begin{proof}
Let $l>0$ and $\{\nball_{f(x_i)}(x_i)\}_{i=1}^l$ be a collection of $l$ disjoint balls all of which intersect $\nball_{f(y)}(y)$ non-trivially.
It follows that for all $i=1,\dotsc,l$, $d^N(x_i,y)\leq f(x_i)+f(y)$, and hence from $f$ being Lipschitz we get
\begin{equation}\label{E:f-bounds}
\frac{1-L}{1+L}f(y)\leq f(x_i)\leq \frac{1+L}{1-L}f(y),
\end{equation}
and hence as $B^N_{f(x_i)}(x_i)\subset B^N_{f(x_i)+d^N(x_i,y)}(y)$ for all $i$, we get $\cup_{i=1}^l\nball_{f(x_i)}(x_i)\subset\nball_{\left(1+2\frac{1+L}{1-L}\right)f(y)}(y)$. Therefore, as the $B^N_{f(x_i)}(x_i)$ are disjoint:
\begin{equation*}
\sum_{i=1}^l\textup{vol}\left(\nball_{f(x_i)}(x_i)\right)\leq\textup{vol}\left(\nball_{\left(1+2\frac{1+L}{1-L}\right)f(y)}(y)\right).
\end{equation*}
It follows from the Taylor expansion of the volume of geodesic balls  combined with \eqref{E:f-bounds} that there exists a constant $c_1 = c_1(n)\in (0,1)$ such that if $f(y)\leq\kappa^{-1}$ satisfies $c_1\left(\frac{1-L}{1+L}f(y)\right)^2\kappa\leq 1$, then
$$\textup{vol}(\nball_{f(x_i)}(x_i))\geq \textup{vol}\left(\nball_{\frac{1-L}{1+L}f(y)}(x_i)\right)\geq \frac{1}{2}\cdot\omega_{n+1}\left(\frac{1-L}{1+L}f(y)\right)^{n+1}$$
for each $i$, where $\omega_n$ is the volume of the unit ball in Euclidean space of dimension $n+1$.
On the other hand, it follows from \cite{carron20}*{Theorem 2.1} that there exists $\bar{R} = \bar{R}(n,\kappa,V)>0$ and $c_2 = c_2(n)\in(0,\infty)$ such that if $r\leq\bar{R}$ then $\textup{vol}(\nball_r(z))\leq c_2 r^{n+1}$ for any $z\in N$.

Now, if $f(y)\leq c_3$, where $c_3:=\min\left\{\kappa^{-1},\frac{1+L}{1-L}(\kappa c_1)^{-1/2},\left(1+2\frac{1+L}{1-L}\right)^{-1}\bar{R}\right\}$, then combining the above we get
\begin{equation*}
l\cdot \frac{1}{2}\omega_{n+1}\left(\frac{1-L}{1+L}\right)^{n+1} f(y)^{n+1}\leq \left(1+2\frac{1+L}{1-L}\right)^{n+1}c_2\cdot f(y)^{n+1}
\end{equation*}
and so the proof is complete after rearranging. If instead $f(y)>c_3$, then:
\begin{equation*}
    \textup{vol}\left(B^N_{\frac{1-L}{1+L}f(y)}(x_i)\right) \geq \textup{vol}\left(B^N_{\frac{1-L}{1+L}c_3}(x_i)\right) \geq \frac{1}{2}\cdot \omega_{n+1}\left(\frac{1-L}{1+L}c_3\right)^{n+1}
\end{equation*}
where in last inequality follows by the same Taylor expansion of the volume of geodesic balls as above (as $c_3$ is sufficiently small). Hence
\begin{align*}
l \cdot \frac{1}{2}\omega_{n+1}\left(\frac{1-L}{1+L}\right)^{n+1}c_3^{n+1} & \leq \sum_{i=1}^l\textup{vol}\left(\nball_{\frac{1-L}{1+L}f(y)}(x_i)\right)\\
&\leq \sum_{i=1}^l\textup{vol}(\nball_{f(x_i)}(x_i))\leq\textup{vol}\left(\nball_{\left(1+2\frac{1+L}{1-L}\right)f(y)}(y)\right)\leq V,
\end{align*}
which concludes the proof in this case as well.
\end{proof}

\bibliographystyle{alpha}
\bibliography{bibliography}
\end{document}